\documentclass{amsart}
\usepackage[margin=1.0in]{geometry}
\usepackage{latexsym, graphicx, amsmath, amssymb, cite, hyperref, xcolor, enumerate, tikz}

\newtheorem{theorem}{Theorem}[section]
\newtheorem{lemma}[theorem]{Lemma}
\newtheorem{proposition}[theorem]{Proposition}
\newtheorem{corollary}[theorem]{Corollary}

\theoremstyle{definition}
\newtheorem{definition}[theorem]{Definition}

\theoremstyle{remark}
\newtheorem{remark}[theorem]{Remark}

\numberwithin{equation}{section}

\newcommand{\R}{{\mathbb R}}

\newcommand{\N}{{\mathbb N}}

\newcommand{\cG}{{\mathcal G}}
\newcommand{\cH}{{\mathcal H}}

\newcommand{\supp}{\operatorname{supp}}
\providecommand{\e}{\varepsilon}
\providecommand{\f}{\varphi}
\providecommand{\x}{\xi}

\title{The VC-dimension and point configurations in ${\Bbb R}^d$}
\author{Alex Iosevich}
\address{Mathematics Department, University of Rochester, Rochester, NY}
\email{alex.iosevich@rochester.edu}

\author{Akos Magyar} 
\address{Mathematics Department, University of Georgia, Athens, GA}
\email{amagyar@uga.edu}

\author{Alex McDonald} 
\address{Mathematics Department, Kennesaw State University, Marietta, GA}
\email{amcdon79@kennesaw.edu}

\author{Brian McDonald}
\address{Mathematics Department, University of Georgia, Athens, GA}
\email{brian.mcdonald@uga.edu}

\thanks{The first listed author was supported in part by the National Science Foundation under grants no. HDR TRIPODS - 1934962, NSF DMS 2154232, and NSF DMS 2506858 during work on this paper.}

\begin{document}

\begin{abstract} Given a set $X$ and a collection ${\mathcal H}$ of functions from $X$ to $\{0,1\}$, the VC-dimension measures the complexity of the hypothesis class $\mathcal{H}$ in the context of PAC learning.  In recent years, this has been connected to geometric configuration problems in vector spaces over finite fields.  In particular, it is easy to show that the VC-dimension of the set of spheres of a given radius in $\mathbb{F}_q^d$ is equal to $d+1$, since this is how many points generically determine a sphere.  It is known, due to \cite{small}, that for $E\subseteq \mathbb{F}_q^d$, $|E|\geq q^{d-\frac{1}{d-1}}$, the set of spheres centered at points in $E$, and intersected with the set $E$, has VC-dimension either $d$ or $d+1$.
\\

\noindent In this paper, we study a similar question over Euclidean space.  We find an explicit dimensional threshold $s_d<d$ so that whenever $E\subseteq \mathbb{R}^d$, $d\geq 3$, and the Hausdorff dimension of $E$ is at least $s_d$, it follows that there exists an interval $I$ such that for any $t\in I$, the VC-dimension of the set of spheres of radius $t$ centered at points in $E$, and intersected with $E$, is at least $3$.  In the process of proving this theorem, we also provide the first explicit dimensional threshold for a set $E\subseteq \mathbb{R}^3$ to contain a $4$-cycle, i.e. $x_1,x_2,x_3,x_4\in E$ satisfying
$$
|x_1-x_2|=|x_2-x_3|=|x_3-x_4|=|x_4-x_1|
$$

\end{abstract}

\maketitle

\section{Introduction} 

\vskip.125in 

The purpose of this paper is to study the Vapnik-Chervonenkis dimension in the context of a naturally arising family of functions on compact subsets of ${\Bbb R}^d$. Let us begin by recalling some definitions and basic results (see e.g. \cite{DS14}, Chapter 6). 

\begin{definition} \label{shatteringdef} Let $X$ be a set and ${\mathcal H}$ a collection of functions from $X$ to $\{0,1\}$. We say that ${\mathcal H}$ shatters a finite set $C \subset X$ if the restriction of ${\mathcal H}$ to $C$ yields every possible function from $C$ to $\{0,1\}$. \end{definition} 

\vskip.125in 

\begin{definition} \label{vcdimdef} Let $X$ and ${\mathcal H}$ be as above. We say that a non-negative integer $n$ is the VC-dimension of ${\mathcal H}$ if there exists a set $C \subset X$ of size $n$ that is shattered by ${\mathcal H}$, and no subset of $X$ of size $n+1$ is shattered by ${\mathcal H}$. \end{definition}

We are going to work with a class of functions ${\mathcal H}^d_t$, where $t \not=0$. Let $X={\Bbb R}^d$, and define 
\begin{equation} \label{functionclassdef} {\mathcal H}_t^d=\{h_y: y \in {\Bbb R}^d \}, \end{equation} where $y \in {\Bbb R}^d$, and $h_y(x)=1$ if $|x-y|=t$, and $0$ otherwise. Let ${\mathcal H}_t^d(E)$ be defined the same way, but with respect to a compact set $E \subset {\Bbb R}^d$ i.e 
$$ {\mathcal H}^d_t(E)=\{h_y: y \in E\},$$ where $h_y(x)=1$ if $|x-y|=t$ ($x \in E$), and $0$ otherwise.  The problem of determining the VC-dimension of $\mathcal{H}_t^d(E)$ amounts to determining whether the set $E$ contains certain finite point configurations.  This question is the subject of much interest in harmonic analysis and geometric measure theory.  The prototypical example of such a problem is the Falconer distance problem, which asks how large the Hausdorff dimension of a compact set $E\subset \R^d$ must be to ensure that the distance set
\[
\Delta(E):=\{|x-y|:x,y\in E\}
\]
has positive Lebesgue measure.  This problem was originated by Falconer \cite{Falc85}, who proved that $\dim E>\frac{d+1}{2}$ was sufficient and conjectured that $\frac{d}{2}$ was the optimal threshold.  The Falconer problem remains open, with the best known results due to Guth, Ou, Wang, and the first author \cite{GIOW} when $d=2$ and to Du, Ou, Ren, Zhang \cite{DORZ} when $d\geq 3$.  The techniques used to study the Falconer problem have proved useful in the study of the existence and abundance of more complex patterns in fractal sets; see, for example, \cite{CIMP,GM22,GILP,GIP17,GIT21,GIT24,GIT25,Areas,PRA23}.  Our main result is the following. 

\begin{theorem} \label{main} Let $d\geq 3$, and define
\[
s_d=\frac{29d+2-\sqrt{81d^2+116d-156}}{20}.
\]
If $E \subset {\Bbb R}^d$ is compact and $\dim E>s_d$, then then there exists a non-degenerate interval $I$ such that for every $t \in I$, the VC-dimension of ${\mathcal H}^d_t(E)$ is at least $3$. 
\end{theorem}
\begin{remark}
By direct calculation, we have $s_d<d-\frac{1}{20}$ for all $d\geq 3$ and $s_d<d-\frac{1}{10}$ for all $d\geq 4$.  For any $\e>0$, we have $s_d<d-\frac{2}{9}+\e$ when $d$ is sufficiently large.  We also note that $s_d>d-1$ for all $d\geq 3$ (this will be used several times in the proofs that follow).
\end{remark}

\begin{remark}
There are a handful of related results in the context of vector spaces over finite fields, with classifiers defined by spheres and hyperplanes, and quadratic residues: \cite{small,small2,quadratic,FIMW24, IMS23, PSTT25}.
    
\end{remark}

We can prove that the VC-dimension is at least $2$ under a weaker assumption. 

\begin{theorem} \label{mainchain} Let $d\geq 2$, and let $E \subset {\Bbb R}^d$ be compact. If $\dim E>\max(\frac{d+1}{2},d-1)$, then there exists a non-degenerate interval $I$ such that for every $t \in I$, the VC-dimension of ${\mathcal H}^d_t(E)$ is at least $2$. \end{theorem} 

It is not hard to verify that the VC-dimension of $\mathcal{H}^d$ is $d+1$.  Therefore, it is natural to ask whether there is a non-trivial dimensional threshold which guarantees that the VC-dimension of $\mathcal{H}_t^d(E)$ is $d+1$.  In $d=2$, the answer is negative.  It is possible for the VC-dimension of ${\mathcal H}^2(E)$ to be equal to $2$ even if the Hausdorff dimension of $E$ is qual to $2$. This is due to the following result of Maga. 
\begin{theorem}[\cite{Maga}, Theorem 2.3]
For any dimension $d\in\N$, there exists a set $E\subset \R^d$ of Hausdorff dimension $d$ which does not contain the verticies of a parallelogram, i.e., a translate of a configuration $\{0,x,y,x+y\}$ with $x,y\neq 0$.
\end{theorem}
Maga's result shows that it is impossible to guarantee $\mathcal{H}_t^2(E)$ shatters a set of $3$ points in $E\subset\R^2$ with dimensional assumptions alone.  Indeed, let $x_1,x_2,x_3\in E$ be distinct.  If $\mathcal{H}_t^2(E)$ shatters $\{x_1,x_2,x_3\}$, then there exists $y_{1,2,3}\in E$ such that $|y_{123}-x_i|=t$ for $i=1,2,3$, and there exists $y_{12}\in E$ such that $|y_{12}-x_i|=t$ for $i=1,2$ and $|y_{12}-x_3|\neq t$.  It follows that $E$ contains the parallelogram $\{x_1,x_2,y_{123},y_{12}\}$.  As a consequence, we have:
\begin{corollary}
There exists a compact set $E\subset \R^2$ of Hausdorff dimension $2$ such that for any $t>0$, the VC-dimension of $\mathcal{H}_t^2(E)$ is at most $2$.
\end{corollary}
In the course of proving Theorem \ref{main}, we also establish the following result on $4$-cycles which is interesting in its own right.
\begin{theorem}
\label{4cycle}
Let $E\subset \R^3$ be compact.  If $\dim(E)>\frac{53-\sqrt{337}}{12}\approx 2.887\dots$, then the set
\[
\{t\in\R: \exists \textup{ distinct }x,y,z,w\in E, |x-y|=|y-z|=|z-w|=|w-x|=t\}
\]
has non-empty interior.
\end{theorem}
\begin{remark}
Greenleaf, Pramanik, and the first listed author \cite{GIP17} proved that the conclusion of Theorem \ref{4cycle} holds when $d\geq 4$ and $E\subset \R^d$ is a set of dimension $\dim E>\frac{d+3}{2}$.  The techniques which prove Theorem \ref{4cycle} extend to higher dimensions, but the result is novel in $d=3$ only.  The existence of a non-trivial threshold follows from the main result of the first two listed authors \cite{IM20}, but Theorem \ref{4cycle} is the first explicit threshold for this configuration in $d=3$.
\end{remark}

\subsection{Notation and conventions}
\label{notation}
The VC-dimension of $\cH_t^d(E)$ is at least $k$ if there exist $x_i,y_I\in E$ for each $i\in \{1,\dots,k\}$ and $I\subset \{1,\dots,k\}$ such that $|x_i-y_I|=t$ if and only if $i\in I$.  We let $\cG_k$ denote the graph corresponding to these distance relations; that is, for $k\in \N$, let $\cG_k$ be the bipartite graph with vertex sets $\{1,\dots,k\}$ and the power set $\mathcal{P}(\{1,\dots,k\})$, and adjacency defined by set membership (i.e., $i\sim I$ if and only if $i\in I$).  \\

Throughout this paper, we will approximate various measures by taking convolutions with approximate identities associated to bump functions.  However, for technical reasons, we will need bump functions with different properties for approximations to different measures.  Fix Schwartz functions $\f,\psi$ on $\R^d$ such that $\f$ is supported in the unit ball, and $\widehat{\psi}$ is supported in the unit ball.  Let $\f^\e(x)=\e^{-d}\f(x/\e)$ be the corresponding approximation to the identity, and define $\psi^\e$ similarly.  For any $t>0$, let $\sigma_t$ be the normalized surface measure on the circle of radius $t$, and let $\sigma_t^\e=\sigma_t*\f^\e$.  For other measures $\mu$ (in practice, Frostman probability measures on compact sets $E\subset \R^d$ with large Hausdorff dimension) let $\mu^\e=\mu*\psi^\e$.

\section{Learning theory perspective on Theorem \ref{main}}

\vskip.125in 

From the point of view of learning theory, it is interesting to ask what the ``learning task" is in the situation at hand. It can be described as follows. We are asked to construct a function $f:E \to \{0,1\}$, $E \subset \R^2$, that is equal to $1$ on a sphere of radius $t$ centered at some $y^{*} \in E$, but we do not know the value of $y^{*}$. The fundamental theorem of statistical learning (see Theorem \ref{FTSL} below) tells us that if the VC-dimension of ${\mathcal H}_t^2(E)$ is finite, we can find an arbitrarily accurate hypothesis (element of ${\mathcal H}^2_t(E)$) with arbitrarily high probability if we consider a randomly chosen sampling training set of sufficiently large size. 

We shall now make these concepts precise (see \cite{DS14} and the references contained therein for more information). Let us recall some more basic notions.

\begin{definition} Given a set $X$, a probability distribution $D$ and a labeling function $f: X \to \{0,1\}$, let $h$ be a hypothesis, i.e $h:X \to \{0,1\}$, and define 
$$ L_{D,f}(h)={\Bbb P}_{x \sim D}[h(x) \not=f(x)],$$ where ${\Bbb P}_{x \sim D}$ means that $x$ is being sampled according to the probability distribution $D$. 
\end{definition} 

\vskip.125in 

\begin{definition} We say that the realizability assumption is satisfied if there exists $h^{*} \in {\mathcal H}$ such that $L_{D,f}(h^{*})=0$. \end{definition} 

\vskip.125in 

\begin{definition} A hypothesis class ${\mathcal H}$ is PAC learnable if there exist a function 
$$m_{{\mathcal H}}: {(0,1)}^2 \to {\Bbb N}$$ and a learning algorithm with the following property: For every $\epsilon, \delta \in (0,1)$, for every distribution $D$ over $X$, and for every labeling function $f: X \to \{0,1\},$ if the realizability assumption holds with respect to $X$, $D$, $f$, then when running the learning algorithm on $m \ge m_{{\mathcal H}}(\epsilon, \delta)$ i.i.d. examples generated by $D$, and labeled by $f$, the algorithm returns
a hypothesis $h$ such that, with probability of at least $1-\delta$ (over the choice of the examples), 
$$L_{D,f}(h) \leq \epsilon.$$
\end{definition} 

\vskip.125in 

The following theorem is a quantitative version of the fundamental theorem of machine learning, and provides the link between VC-dimension and learnability (see \cite{DS14}).

\begin{theorem} \label{FTSL} Let ${\mathcal H}$ be a collection of hypotheses on a set $X$. Then ${\mathcal H}$ has a finite VC-dimension if and only if ${\mathcal H}$ is PAC learnable. Moreover, if the VC-dimension of ${\mathcal H}$ is equal to $n$, then ${\mathcal H}$ is PAC learnable and there exist constants $C_1, C_2$ such that 
$$ C_1 \frac{n+\log \left(\frac{1}{\delta} \right)}{\epsilon} \leq m_{{\mathcal H}}(\epsilon, \delta) \leq C_2 \frac{n \log \left(\frac{1}{\epsilon} \right)+
\log \left(\frac{1}{\delta} \right)}{\epsilon}.$$

\end{theorem}

\section{Configuration integrals corresponding to finite graphs}
\label{configurationintegrals}
Let $G$ be a graph on vertices $\{0,1,\dots,n\}$ with adjacency relation $\sim$.  For a vertex $j$ and points $x_0,\dots,x_n\in \R^d$, let $A_j(x_0,\dots,x_{j-1})$ denote the affine span of $\{x_i:i\sim j,j<i\}$.  For $c>0$, define
\begin{equation}
\label{Nc}
N_c=\{(x_0,\dots,x_n)\in (\R^d)^{n+1}:\text{dist}(x_j,A_j(x_0,\dots,x_{j-1}))>c \ \ \forall j\}.
\end{equation}

For any probability measure $\mu$ on $\R^d$, define
\begin{equation}
\label{Lambda}
\Lambda_{G,t,c}^\e\mu:=\int_{N_c} \left(\prod_{i\sim j} \sigma_t^\e(x_i-x_j)\right)d\mu^{n+1}(x_0,\dots x_n).
\end{equation}
For small values of $\e$, this integral is a weighted count of the point configurations $(x_0,\dots,x_n)$, consisting of points in the support of $\mu$ which are non-degenerate in a quantitative sense, and such that $|x_i-x_j|=t$ whenever the points $x_i,x_j$ correspond to adjacent vertices $i\sim j$ in our graph.  The following lemma makes this idea precise.
\begin{lemma}
\label{Borelmeasurelemma}
Let $\mu$ be a measure supported on a compact set $E\subset \R^d$, and suppose $G,t,c$ are such that $\Lambda_{G,t,c}^\e\mu\approx 1$, with constants independent of $\e$.  Then, as $\e\to 0$, $\Lambda_{G,t,c}^{\e} \mu$ has a weak* limit $\Lambda_{G,t,c}\mu$ which is a Borel probability measure supported on the set
\[
\{(x_0,\dots,x_n)\in E^{n+1}\cap N_c:|x_i-x_j|=t \text{ for all } i\sim j\}.
\]
\end{lemma}
\begin{proof}
Let $X\subset (\R^d)^{n+1}$ be a compact set containing a neighborhood of $E^{k+1}$, and let $\mathcal{M}(X)$ be the space of complex measures on $K$.  For sufficiently small $\e$, identify $\Lambda_{G,t,c}^\e\mu$ with the measure in $\mathcal{M}(X)$ defined as follows.  For $A\subset X$, let $A_\e$ denote the $\e$-neighborhood of $A$.  Let $\Lambda_{G,t,c}^\e\mu(A)$ denote the integral in (\ref{Lambda}), modified to integrate over $N_c\cap A_\e$ instead of just $N_c$.  The bound $\Lambda_{G,t,c}^\e\mu\lesssim 1$ implies that this family of measures is bounded in the total variation norm.  Since $\mathcal{M}(X)$ is dual to the space of continuous functions on $K$, and this space is separable, the sequential version of the Banach-Alaoglu Theorem implies it has a weak* limit $\Lambda_{G,t,c}\mu$.  The bound $\Lambda_{G,t,c}^\e\mu\gtrsim 1$ implies this measure can be normalized so that it is a probability measure.  It is easy to verify that the support of this measure is as claimed.
\end{proof}
\begin{proposition}[\cite{IT19}, Lemma 2.1]
\label{IT19}
Let $E\subset \R^d$ be compact, and let $\mu$ be a Frostman probability measure supported on $E$ with exponent $s>\frac{d+1}{2}$.  There exists $E_0\subset E$ and a non-degenerate interval $I\subset \R$ such that $\mu(E_0)>0$, and for all $\e>0$ sufficiently small, $t\in I$, and $x\in E_0$, we have $\sigma_t^\e*\mu(x)\approx 1$.
\end{proposition}

\begin{lemma}[Deforestation lemma]
\label{deforestation}
Let $G$ be a finite graph on $\{0,\dots,n\}$, and suppose the vertices are ordered such that $\deg(i)\geq 2$ if and only if $i\leq n_0$.  Let $G_0$ be the graph on vertices $\{0,\dots,n_0\}$ obtained by deleting vertices $n_0+1,\dots,n$ from $G$, along with their incident edges.  Suppose further that $\mu$ is a probability measure, and $t,\e$ are such that $\sigma_t^\e*\mu\approx 1$ on the support of $E$.  Then,
\[
\Lambda_{G,t,c}^\e\mu\approx \Lambda_{G_0,t,c}^\e\mu.
\]
\end{lemma}
\begin{proof}
Without loss of generality, suppose we have ordered our vertices so that $\deg(i)=1$ for $i=n_0+1,\dots,n_1$ and $\deg{i}=0$ for $i=n_1+1,\dots,n$.  In the $\deg(i)=1$ case, let $j(i)$ be the unique vertex adjacent to $i$.  Then, we have
\begin{align*}
\Lambda_{G,t,c}^\e\mu &=\int_{N_c} \left(\prod_{i\sim j} \sigma_t^\e(x_i-x_j)\right)d\mu^{n+1}(x_0,\dots x_n) \\
&=\int_{N_c} \left(\prod_{\substack{i\sim j \\ i,j\leq n_0}} \sigma_t^\e(x_i-x_j)\right)\left(\prod_{i=n_0+1}^{n_1} \sigma_t^\e(x_i-x_{j(i)})\right)d\mu^{n_1+1}(x_0,\dots x_{n_1}),
\end{align*}
where we have used the fact that $\mu$ is a probability measure to integrate out $x_{n_1+1},\dots,x_n$.  Next we run the integrals in $x_{n_0+1},\dots,n_{n_1}$, giving
\begin{align*}
\Lambda_{G,t,c}^\e\mu&=\int_{N_c} \left(\prod_{\substack{i\sim j \\ i,j\leq n_0}} \sigma_t^\e(x_i-x_j)\right)\left(\prod_{i=n_0+1}^{n_1} \sigma_t^\e*\mu(x_{j(i)})\right)d\mu^{n_0+1}(x_0,\dots x_{n_0}) \\
&\approx \int_{N_c} \left(\prod_{\substack{i\sim j \\ i,j\leq n_0}} \sigma_t^\e(x_i-x_j)\right)d\mu^{n_0+1}(x_0,\dots x_{n_0}) \\
&=\Lambda_{G_0,t,c}^\e\mu.
\end{align*}
\end{proof}
To prove the bounds we need, we will also need an alternate but equivalent configuration integral.  Fix a graph $G$ with the same notation as above, and for $j=1,\dots,n$, let $V_j=\{i: i<j, i\sim j\}$.  We will always assume that $V_j$ is non-empty for every $j$.  Given $u_1,\dots,u_{j-1}\in\R^d$, let $\sigma_{t;u_1,\dots,u_{j-1}}$ be the normalized surface measure on the surface
\[
S_{t;u_1,\dots,u_{j-1}}:=\{u\in \R^d:|u-u_i|=t \text{ for all } i\in V_j\},
\]
with the convention $u_0=0$.  Let $A_j(u_1,\dots,u_{j-1})$ denote the affine span of $\{u_i:i\in V_j\}$.  Given a constant $c>0$ and $j=0,\dots,n$, let $\eta_{c,j}:(\R^d)^{j+1}\to \R$ be a smooth bump function satisfying $0\leq \eta_{c,j}\leq 1$ and
\[
\eta_{c,j}(u_1,\dots,u_j)=
\begin{cases}
0, & \text{dist}(u_j,A_j(u_1,\dots,u_{j-1}))<c \\
1, & \text{dist}(u_j,A_j(u_1,\dots,u_{j-1}))>2c.
\end{cases}
\]
Let 
\[
d\omega_{t,c,j}(u_1,\dots,u_j)=\eta_{c,j}(u_1,\dots,u_j)d\sigma_{t;u_1,\dots,u_{j-1}}(u_j)\dots d\sigma_{t;u_1}(u_2)d\sigma_t(u_1).
\]
The set $S_{t;u_1,\dots,u_{j-1}}$ is the intersection of $|V_j|$-many spheres in $\R^d$; as a result, we have the following estimate.

Let $\eta_c=\eta_{c,n}$ and $\omega_{c,t}=\omega_{c,t,n}$.  Given a graph $G$ and parameters $c,t>0$, we define the following associated multilinear operator:
\begin{equation}
\label{T}
T_{G,t,c}(f_0,f_1,\dots,f_n)=\int f_0(x)f_1(x-u_1)\dots f_n(x-u_n)d\omega_{c,t}(u_1,\dots,u_n)dx.
\end{equation}
Let $T_{G,t,c}f=T_{G,t,c}(f,f,\dots,f)$. 
\begin{lemma}
\label{equivalence}
For any finite graph $G$ and parameters $t,c,\e>0$, we have
\[
T_{G,c,t}(f_0,f_1,\dots,f_n)\approx \int_{N_c} \left(\prod_{i\sim j} \sigma_t^\e(x_i-x_j)\right)f_0(x_0)f_1(x_1)\cdots f_n(x_n)dx_0dx_1\cdots dx_n.
\]
In particular, if $\mu$ is a probability measure, we have $\Lambda_{G,t,c}^\e\mu\approx T_{G,t,c}\mu^\e$.
\end{lemma}
\begin{proof}
Write $\widetilde{u}=(u_1,\dots,u_{n-1})$, and let
\[
F(x,\widetilde{u})=f_0(x)f_1(x-u_1)\dots f_{n-1}(x-u_{n-1}).
\]
Then, the expression for $T_{G,t,c}(f_0,\dots,f_n)$ simplifies to
\begin{align*}
&\int\int\int_{\text{dist}(u_n,A_n(0,u_1,\dots,u_n))>c} F(x,\widetilde{u})f_n(x-u_n)d\sigma_{\widetilde{u}}(u_n)d\omega_{c,t,n-1}(\widetilde{u})dx \\
\approx &\int\int\int_{\text{dist}(u_n,A_n(0,u_1,\dots,u_n))>c} F(x,\widetilde{u})\left(\prod_{i\sim n}\sigma_t^\e(u_n-u_i)\right)f_n(x-u_n)du_nd\omega_{c,t,n-1}(\widetilde{u})dx \\
=&\int\int\int_{\text{dist}(u_n,A_n(0,u_1,\dots,x-x_n))>c} F(x,\widetilde{u})\left(\prod_{i\sim n}\sigma_t^\e((x-x_n)-u_i)\right)f_n(x_n)dx_nd\omega_{c,t,n-1}(\widetilde{u})dx,
\end{align*}
where the last line is obtained by the linear change of variables $x_n=x-u_n$.  Continuing in this way for each variable, we obtain the claim.
\end{proof}

\section{Bounds on the configuration integral of the $4$-cycle}
\label{4cyclesection}
Throughout the remainder of the paper, we let $\Gamma$ denote the $4$-cycle graph.  We make the convention that the vertex set is $\{0,1,2,3\}$ and the adjacency relation is given by $0\sim 1\sim 3\sim 2\sim 0$.

\begin{figure}[h!]
\begin{center}
\begin{tikzpicture}[scale=2, thick]
  \node[fill=black, circle, minimum size=4pt, inner sep=0pt, label=below left:$0$] (0) at (0,0) {};
  \node[fill=black, circle, minimum size=4pt, inner sep=0pt, label=below right:$1$] (1) at (1,0) {};
  \node[fill=black, circle, minimum size=4pt, inner sep=0pt, label=above right:$2$] (2) at (1,1) {};
  \node[fill=black, circle, minimum size=4pt, inner sep=0pt, label=above left:$3$] (3) at (0,1) {};

  \draw (0) -- (1);
  \draw (1) -- (2);
  \draw (2) -- (3);
  \draw (3) -- (0);
\end{tikzpicture}
\end{center}
\caption{4-cycle graph $\Gamma$}\label{4cycle}
\end{figure}
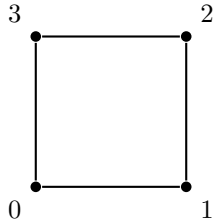

In this section, we fix $t,c>0$ and consider the operator $T=T_{\Gamma,t,c}$ corresponding to this graph.  Thus, we have
\[
T(f_0,f_1,f_2,f_3)=\int\int\int\int f_0(x)f_1(x-u_1)f_2(x-u_2)f_3(x-u_3)\eta(u_1,u_2,u_3)d\sigma_{u_1,u_2}(u_3)d\sigma(u_1)d\sigma(u_2)dx,
\]
where $\eta=\eta_c$ is the bump function in Section \ref{configurationintegrals}, $\sigma$ is the normalized surface measure on the sphere of radius $t$, and $\sigma_{u_1,u_2}$ is the normalized surface measure on the set 
\[
\{u_3:|u_1-u_3|=|u_2-u_3|=t\}.
\]
If $u_3\in\supp\sigma_{u_1,u_2}$, then $(u_1,u_2,u_3)\in\supp \eta$ if and only if $|u_1-u_2|<2\sqrt{t^2-c^2}$, by the picture in Figure \ref{triangle}:

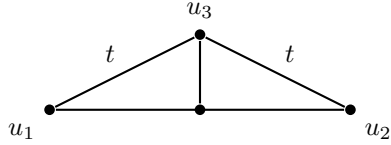
\begin{figure}[h!]
\begin{center}
\begin{tikzpicture}[scale=2, thick]
  \node[fill=black, circle, minimum size=4pt, inner sep=0pt, label=below left:$u_1$] (u_1) at (0,0) {};
  \node[fill=black, circle, minimum size=4pt, inner sep=0pt, label=below right:$u_2$] (u_2) at (2,0) {};
  \node[fill=black, circle, minimum size=4pt, inner sep=0pt, label=above:$u_3$] (u_3) at (1,.5) {};
  \node[fill=black, circle, minimum size=4pt, inner sep=0pt] (m) at (1,0) {};

  \draw (u_1) -- (u_2);
  \draw (u_1) -- (u_3) node[midway, above left] {$t$};
  \draw (u_2) -- (u_3) node[midway, above right] {$t$};
  \draw (m) -- (u_3);
\end{tikzpicture}
\end{center}
\caption{$\text{dist}(u_3,A_3(u_1,u_2))>c$}\label{triangle}
\end{figure}

Therefore, we have the simpler expression
\[
T(f_0,f_1,f_2,f_3)=\int\int\int\int_{|u_1-u_2|<2\sqrt{t^2-c^2}} f_0(x)f_1(x-u_1)f_2(x-u_2)f_3(x-u_3)d\sigma_{u_1,u_2}(u_3)d\sigma(u_1)d\sigma(u_2)dx.
\]

We will need the following calculation.
\begin{lemma}
\label{sphereintersection}
For any $\xi\in \R^d\setminus \{0\}$, we have
\[
\int\int_{|u_1-u_2|<2\sqrt{t^2-c^2}} |\widehat{\sigma}_{u_1,u_2}(\xi)|^2d\sigma(u_1)d\sigma(u_2)\lesssim_{t,c} (1+|\xi|)^{-(d-2)}.
\]
\end{lemma}
\begin{proof}
Let $V_{u_1,u_2}$ denote the orthogonal complement of the vector $u_1-u_2$.  The measure $\sigma_{u_1,u_2}$ is the normalized surface measure on the intersection of spheres (of radius $t$) centered at $u_1$ and $u_2$.  This intersection is a $(d-2)$-dimensional sphere contained in a hyperplane parallel to $V_{u_1,u_2}$.  Moreover, the radius of this $(d-2)$-sphere is bounded below, with bound depending on $t$ and $c$ but not on $u_1$ or $u_2$.  If $P_{u_1,u_2}:\R^d\to V_{u_1,u_2}$ is the orthogonal projection onto this hyperplane, then
\[
|\widehat{\sigma}_{u_1,u_2}(\xi)|\lesssim_{t,c} (1+|P_{u_1,u_2}(\xi)|)^{-\frac{d-2}{2}}.
\]
If $\theta_{u_1,u_2}(\xi)$ is the angle between the vectors $\xi$ and $u_1-u_2$, then 
\[
|P_{u_1,u_2}(\xi)|=|\xi|\sin\theta_{u_1,u_2}(\xi).
\]
For fixed $u_2$ and $\xi$, let $U_j(u_2,\xi)=\{u_1:2^{-(j+1)}<\theta_{u_1,u_2}(\xi)\leq 2^{-j}\}$.  Then, $U_j(u_1,\xi)$ is a subset of the $(d-1)$-sphere of diameter $\approx 2^{-j}$, hence $\sigma(U_j(u_2,\xi))\lesssim 2^{-j(d-1)}$ and
\begin{align*}
\int\int |\widehat{\sigma}_{u_1,u_2}(\xi)|^2d\sigma(u_1)d\sigma(u_2) &\lesssim_{t,c} \sum_{j=0}^\infty (1+2^{-j}|\xi|)^{-(d-2)}\int\int_{U_j(u_1,\xi)} d\sigma(u_1)d\sigma(u_2) \\
&\lesssim \sum_{j=0}^\infty (1+|\xi|)^{-(d-2)}2^{j(d-2)}2^{-j(d-1)} \\
&\approx (1+|\xi|)^{-(d-2)}.
\end{align*}
\end{proof}
\begin{lemma}
\label{LinfL1}
Let $d\geq 3$.  Suppose $\mu$ is a Frostman probability measure with exponent $s<d$, supported on a compact subset of $\R^d$, with finite energy integral:
\begin{equation}
\label{energy}
I_s(\mu):=\int |\widehat{\mu}(\xi)|^2|\xi|^{-(d-s)}d\xi<\infty.
\end{equation}
Let $\Delta_\e=\mu_{2\e}-\mu_\e$, and let $f\in L^\infty$.  Then, for any $t,c>0$, we have
\[
|T_{t,c}(\mu_\e,\mu_\e,f,\Delta_\e)|\lesssim_{t,c} \e^{\left(-2(d-s)+\frac{(d+2-s)(d-2)}{2(2d-s)}\right)}\|f\|_{L^\infty}.
\]
\end{lemma}
\begin{proof}
Fix $t,c>0$, and let $T=T_{t,c}$.  For simplicity, all implicit constants will be allowed to depend on $t$ and $c$, so $\lesssim_{t,c}$ will be written simply as $\lesssim$.  Using the bounds $\|\mu_\e\|_{L^\infty}\lesssim \e^{-(d-s)}$ and $\|\mu_\e\|_{L^2}\lesssim \e^{-(d-s)/2}$, we have
\begin{align*}
|T(\mu_\e,\mu_\e,f,\Delta_\e)|&=\left|\int\int\int\int_{|u_1-u_2|<2\sqrt{t^2-c^2}} \mu_\e(x)\mu_\e(x-u_1)f(x-u_2)\Delta_\e(x-u_3)d\sigma_{u_1,u_2}(u_3)d\sigma(u_1)d\sigma(u_2)dx\right| \\
&=\left|\int\int\int_{|u_1-u_2|<2\sqrt{t^2-c^2}} \mu_\e(x)\mu_\e(x-u_1)f(x-u_2)\Delta_\e*\sigma_{u_1,u_2}(x)d\sigma(u_1)d\sigma(u_2)dx\right| \\
&\lesssim \e^{-(d-s)}\|f\|_{L^\infty}\int\int\int_{|u_1-u_2|<2\sqrt{t^2-c^2}} |\mu_\e(x)\Delta_\e*\sigma_{u_1,u_2}(x)|dxd\sigma(u_1)d\sigma(u_2) \\
&\lesssim \e^{-\frac{3}{2}(d-s)}\|f\|_{L^\infty} \int\int_{|u_1-u_2|<2\sqrt{t^2-c^2}} \|\Delta_\e*\sigma_{u_1,u_2}\|_{L^2}d\sigma(u_1)d\sigma(u_2) \\
&= \e^{-\frac{3}{2}(d-s)}\|f\|_{L^\infty} \int\int_{|u_1-u_2|<2\sqrt{t^2-c^2}} \|\widehat{\Delta}_\e\widehat{\sigma}_{u_1,u_2}\|_{L^2}d\sigma(u_1)d\sigma(u_2).
\end{align*}
By applying Cauchy-Schwarz to the last integral, we have
\begin{align*}
|T(\mu_\e,\mu_\e,f,\Delta_\e)|^2&\lesssim \e^{-3(d-s)}\|f\|_{L^\infty}^2\int\int_{|u_1-u_2|<2\sqrt{t^2-c^2}} \|\widehat{\Delta}_\e\widehat{\sigma}_{u_1,u_2}\|_{L^2}^2d\sigma(u_1)d\sigma(u_2) \\
&=\e^{-3(d-s)}\|f\|_{L^\infty}^2\int |\widehat{\Delta}_\e(\xi)|^2\left(\int\int_{|u_1-u_2|<2\sqrt{t^2-c^2}} |\widehat{\sigma}_{u_1,u_2}(\xi)|^2d\sigma(u_1)d\sigma(u_2)\right)d\xi.
\end{align*}
By Lemma \ref{sphereintersection}, this is
\begin{equation}
\label{Deltaintegral}
|T(\mu_\e,\mu_\e,f,\Delta_\e)|^2\lesssim \e^{-3(d-s)}\|f\|_{L^\infty}^2\int|\widehat{\Delta}_\e(\xi)|^2(1+|\xi|)^{-(d-2)}d\xi.
\end{equation}
Recall that 
    \[
    \Delta_\e=\mu_{2\e}-\mu_\e=\mu*(\psi_{2\e}-\psi_\e),
    \]
    hence
    \[
    |\widehat{\Delta_\e}(\xi)|^2=|\widehat{\mu}(\x)|^2|\widehat{\psi}(2\e\x)-\widehat{\psi}(\e\x)|^2.
    \]
We chose $\psi$ so that $\widehat{\psi}$ is supported on the unit ball, so $\widehat{\Delta}_\e$ is supported on $|\xi|<\e^{-1}$.  We bound the integral separately in the regions $|\x|<\e^{-\alpha}$ and $\e^{-\alpha}<|\xi|<\e^{-1}$, optimizing the parameter $\alpha$ later.  By (\ref{energy}), we have
\[
\int_{|\xi|<\e^{-1}}|\widehat{\mu}(\xi)|^2d\xi  =\int_{|\xi|<\e^{-1}}|\widehat{\mu}(\xi)|^2|\xi|^{-(d-s)}|\xi|^{d-s}d\xi \lesssim I_s(\mu)\e^{-(d-s)} \lesssim \e^{-(d-s)}.
\]
Since $\widehat{\psi}$ is Schwarz and therefore Lipschitz, we also have $|\widehat{\psi}(2\e\x)-\widehat{\psi}(\e\x)|\lesssim \e\x$.  Therefore,
    \begin{align*}
&\int_{|\x|<\e^{-\alpha}}|\widehat{\Delta}_\e(\xi)|^2(1+|\xi|)^{-(d-2)}d\xi \\
    \lesssim &\int_{|\x|<\e^{-\alpha}} |\widehat{\Delta_\e}(\x)|^2d\x \\
    \lesssim &\e^{2(1-\alpha)}\int_{|\x|<\e^{-\alpha}}|\widehat{\mu}(\x)|^2d\x \\
    \lesssim &\e^{2(1-\alpha)-\alpha(d-s)}
    \end{align*}
and
\begin{align*}
&\int_{\e^{-\alpha}<\xi<\e^{-1}}|\widehat{\Delta}_\e(\xi)|^2(1+|\xi|)^{-(d-2)}d\xi \\
    \lesssim &\e^{\alpha(d-2)}\int_{\e^{-\alpha}<\xi<\e^{-1}} |\widehat{\Delta_\e}(\x)|^2d\x \\
    \lesssim &\e^{\alpha(d-2)}\int_{\e^{-\alpha}<\xi<\e^{-1}}|\widehat{\mu}(\x)|^2|\widehat{\f}(2\e\x)-\widehat{\f}(\e\x)|^2d\x \\
\lesssim &\e^{\alpha(d-2)}\int_{\xi<\e^{-1}}|\widehat{\mu}(\x)|^2d\x \\
    \lesssim &\e^{\alpha(d-2)-(d-s)}.
\end{align*}
Setting $\alpha=\frac{d+2-s}{2d-s}$ and plugging the resulting bound into (\ref{Deltaintegral}) gives
\[
|T(\mu_\e,\mu_\e,f,\Delta_\e)|^2\lesssim_{t,c}\e^{-3(d-s)}\|f\|_{L^\infty}^2\e^{\frac{(d+2-s)(d-2)}{2d-s}-(d-s)},
\]
or
\[
|T(\mu_\e,\mu_\e,f,\Delta_\e)|\lesssim_{t,c} \e^{-2(d-s)+\frac{(d+2-s)(d-2)}{2(2d-s)}}\|f\|_{L^\infty}.
\]
\end{proof}
We will use Lemma \ref{LinfL1} directly in our proof of Lemma \ref{mainupperbound}, which is the upper bound used in the proof of Theorem \ref{main}.  It also gives the following bound for the configuration integral of the $4$-cycle, which is used in the proof of Theorem \ref{4cycle}.
\begin{lemma}
\label{4cycleupperbound}
Let $d\geq 3$, let $t,c>0$, and let $\mu$ be a compactly supported Frostman probability measure with exponent $s$, with finite energy integral (\ref{energy}).  If
\begin{equation}
\label{4cycleexponent}
s>\frac{17d+2-\sqrt{25d^2+68d-92}}{12},
\end{equation}
then $\Lambda_{\Gamma,t,c}\lesssim 1$.
\end{lemma}
\begin{remark}
The threshold in (\ref{4cycleexponent}) is smaller than the quantity $s_d$ defined in the statement of Theorem \ref{main}.
\end{remark}
\begin{proof}
Let $\Delta_\e$ be as defined in the statement of Lemma \ref{LinfL1}.  We use the telescoping identity
\begin{align*}
T\mu_{2\e}-T\mu_\e=&T(\Delta_\e,\mu_\e,\mu_\e,\mu_\e) \\
&+T(\mu_{2\e},\Delta_\e,\mu_\e,\mu_\e) \\
&+T(\mu_{2\e},\mu_{2\e},\Delta_\e,\mu_\e) \\
&+T(\mu_{2\e},\dots,\mu_{2\e},\Delta_\e).
\end{align*}
By symmetry, each term satisfies the bound in Lemma \ref{LinfL1} with $f=\mu_\e$, giving
\[
|T\mu_{2\e}-T\mu_\e|\lesssim \e^{-3(d-s)+\frac{(d+2-s)(d-2)}{2(2d-s)}}.
\]
If $s$ satisfies (\ref{4cycleexponent}) then the exponent of $\e$ above is positive, hence the telescoping series
\[
|T\mu_\e|\leq\sum_{j=0}^\infty |T\mu_{2^{-j+1}\e}-T\mu_{2^{-j}\e}|
\]
converges.  By Lemma \ref{equivalence}, it follows that $\Lambda_{\Gamma,t,c}\lesssim 1$.
\end{proof}
\section{Proofs of Theorems \ref{main}, \ref{4cycle}, and \ref{mainchain}}
\label{mainproof}
\subsection{Initial reductions and outline of section}

Let $\cG_3$ be the graph corresponding to three shattered points as defined in Section \ref{notation}.  Because we are going to look at configurations in $\R^d$ corresponding to this graph, it will be useful to use slightly different notation than is used in the definition.  Let $\{x_1,x_2,x_3\}$ be the points which are being shattered, and let
\[
\{y_\varnothing,y_1,y_2,y_3,y_{12},y_{23},y_{13},y_{123}\}
\]
be the centers of the circles.  In this notation, $\mathcal{G}_3$ is as shown in Figure \ref{shatter3}.  Let $G$ be the subgraph obtained from $\mathcal{G}_3$ by removing vertices of degree 0 and 1, and let $H$ be the subgraph of $G$ with vertices $\{x_1,x_2,y_{12},y_{13},y_{123}\}$; again, these are shown in Figure \ref{shatter3}.  For the remainder of this section, we will work with configuration integrals defined with respect to the graph $G$.  Let $E$ be a compact set, and let $\mu$ be a Frostman probability measure on $E$ with exponent $s<\dim E$.  Our goal is to obtain upper and lower bounds on the configuration integrals $T_{G,c,t}\mu_\e$ which do not depend on $\e$.  In Subsection \ref{upperbound}, we prove the upper bound.  The starting point is a technique introduced by the first two authors in \cite{IM20}.  This does not apply directly to the graphs under consideration, so we combine this approach with an interpolation argument which exploits the symmetry of the graph $\mathcal{G}_3$.  In Subsection \ref{lowerbound}, we obtain a lower bound using novel techniques which exploit the symmetry of the graph $G$.  In Subsection \ref{proofmt}, we verify that ``degenerate'' configurations do not dominate our estimates, and complete the proof of the theorem. 

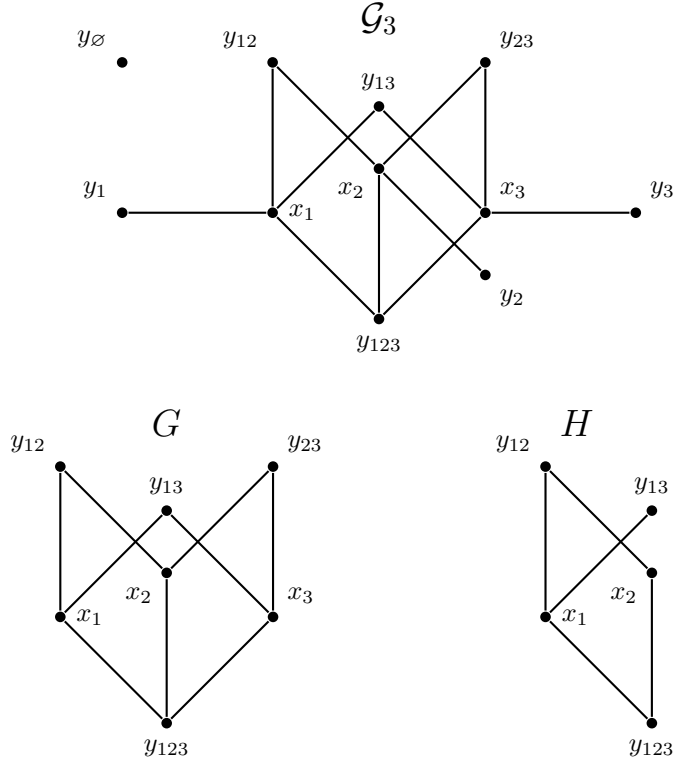
\begin{figure}[h!]
\centering
\begin{minipage}{\textwidth}
  \centering
\begin{tikzpicture}[scale=2, thick]
  \node[fill=black, circle, minimum size=4pt, inner sep=0pt, label=below:$y_{123}$] (y123) at (0,0) {};
  \node[fill=black, circle, minimum size=4pt, inner sep=0pt, label=below left:$x_2$] (x2) at (0,1) {};
  \node[fill=black, circle, minimum size=4pt, inner sep=0pt, label=right:$x_1$] (x1) at (-.707,.707) {};
  \node[fill=black, circle, minimum size=4pt, inner sep=0pt, label=above right:$x_3$] (x3) at (.707,.707) {};
  \node[fill=black, circle, minimum size=4pt, inner sep=0pt, label=above left:$y_{12}$] (y12) at (-.707,1.707) {};
  \node[fill=black, circle, minimum size=4pt, inner sep=0pt, label=above right:$y_{23}$] (y23) at (.707,1.707) {};
  \node[fill=black, circle, minimum size=4pt, inner sep=0pt, label=above:$y_{13}$] (y13) at (0,1.414) {};
  \node[fill=black, circle, minimum size=4pt, inner sep=0pt, label=above left:$y_1$] (y1) at (-1.707,.707) {};
  \node[fill=black, circle, minimum size=4pt, inner sep=0pt, label=above right:$y_3$] (y3) at (1.707,.707) {};
  \node[fill=black, circle, minimum size=4pt, inner sep=0pt, label=below right:$y_2$] (y2) at (.707,.293) {};
  \node[fill=black, circle, minimum size=4pt, inner sep=0pt, label=above left:$y_\varnothing$] (y0) at (-1.707,1.707) {};
  \node at (0,2) {\LARGE $\mathcal{G}_3$};

  \draw (y123) -- (x1);
  \draw (y123) -- (x2);
  \draw (y123) -- (x3);
  \draw (y12) -- (x1);
  \draw (y12) -- (x2);
    \draw (y13) -- (x1);
    \draw (y13) -- (x3);
    \draw (y23) -- (x2);
    \draw (y23) -- (x3);
    \draw (y1) -- (x1);
    \draw (y2) -- (x2);
    \draw (y3) -- (x3);
\end{tikzpicture}
\end{minipage}%
\\
\vspace{.5cm}
\begin{minipage}{.33\textwidth}
  \centering
\begin{tikzpicture}[scale=2, thick]
  \node[fill=black, circle, minimum size=4pt, inner sep=0pt, label=below:$y_{123}$] (y123) at (0,0) {};
  \node[fill=black, circle, minimum size=4pt, inner sep=0pt, label=below left:$x_2$] (x2) at (0,1) {};
  \node[fill=black, circle, minimum size=4pt, inner sep=0pt, label=right:$x_1$] (x1) at (-.707,.707) {};
  \node[fill=black, circle, minimum size=4pt, inner sep=0pt, label=above right:$x_3$] (x3) at (.707,.707) {};
  \node[fill=black, circle, minimum size=4pt, inner sep=0pt, label=above left:$y_{12}$] (y12) at (-.707,1.707) {};
  \node[fill=black, circle, minimum size=4pt, inner sep=0pt, label=above right:$y_{23}$] (y23) at (.707,1.707) {};
  \node[fill=black, circle, minimum size=4pt, inner sep=0pt, label=above:$y_{13}$] (y13) at (0,1.414) {};  \node at (0,2) {\LARGE $G$};

  \draw (y123) -- (x1);
  \draw (y123) -- (x2);
  \draw (y123) -- (x3);
  \draw (y12) -- (x1);
  \draw (y12) -- (x2);
    \draw (y13) -- (x1);
    \draw (y13) -- (x3);
    \draw (y23) -- (x2);
    \draw (y23) -- (x3);

\end{tikzpicture}
\end{minipage}
\begin{minipage}{.33\textwidth}
  \centering
\begin{tikzpicture}[scale=2, thick]
  \node[fill=black, circle, minimum size=4pt, inner sep=0pt, label=below:$y_{123}$] (y123) at (0,0) {};
  \node[fill=black, circle, minimum size=4pt, inner sep=0pt, label=below left:$x_2$] (x2) at (0,1) {};
  \node[fill=black, circle, minimum size=4pt, inner sep=0pt, label=right:$x_1$] (x1) at (-.707,.707) {};

  \node[fill=black, circle, minimum size=4pt, inner sep=0pt, label=above left:$y_{12}$] (y12) at (-.707,1.707) {};

  \node[fill=black, circle, minimum size=4pt, inner sep=0pt, label=above:$y_{13}$] (y13) at (0,1.414) {};
    \node at (-.5,2) {\LARGE $H$};

  \draw (y123) -- (x1);
  \draw (y123) -- (x2);

  \draw (y12) -- (x1);
  \draw (y12) -- (x2);
    \draw (y13) -- (x1);

\end{tikzpicture}
\end{minipage}
\caption{The graphs $\cG_3, G, H$.}
  \label{shatter3}
\end{figure}

\subsection{The upper bound}
\label{upperbound}
Throughout this subsection, we fix $c,t>0$ and suppress them from the notation (for example, writing $T_G$ instead of $T_{G,c,t}$).  Let $G$ and $H$ be as defined above.  We will also use the notation $s_d$ as defined in the statement of Theorem \ref{main}.
\begin{lemma}
\label{mainupperbound}
Let $d\geq 3$, and let $\mu$ be a Frostman probability measure on with exponent $s$, supported on a compact subset of $\R^d$.  Then,
\[
|T_G\mu_{2\e}-T_G\mu_\e|\lesssim \e^{-5(d-s)+\frac{(d+2-s)(d-2)}{2(2d-s)}}.
\]
If $s>s_d$, then $|T_G\mu_\e|\lesssim 1$.
\end{lemma}
\begin{proof}
Let $\Delta_\e=\mu_{2\e}-\mu_\e$.  We use the telescoping identity
\begin{align*}
T_G\mu_{2\e}-T_G\mu_\e=&T_G(\Delta_\e,\mu_\e,\dots,\mu_\e) \\
&+T_G(\mu_{2\e},\Delta_\e,\mu_\e,\dots,\mu_\e) \\
&+\cdots \\
&+T_G(\mu_{2\e},\dots,\mu_{2\e},\Delta_\e).
\end{align*}
We will bound the last term on the right (all cases are similar).  Define a measure $\nu$ by
\[
d\nu(x,u_1,u_2)=\mu_{2\e}(x)\mu_{2\e}(x-u_1)\mu_{2\e}(x-u_2)dx d\sigma(u_1) du_2.
\]
Let $\sigma_{u_2,u_3}$ be as in Section\ref{4cyclesection} (note that the numbering of the vertices has changed), and define a linear operator $J$ by
\[
Jf(x,u_1,u_2)=\int\int f(x-u_3)\Delta_\e(x-u_4)d\sigma_{u_1,u_3}(u_4)d\sigma_{u_2}(u_3).
\]
The motivation for these choices is that we have
\begin{equation}
\label{L2TG}
\|Jf\|_{L^2(\nu)}^2=T_G(\mu_{2\e},\dots,\mu_{2\e},f,\Delta_\e)
\end{equation}
and
\begin{equation}
\label{L1TH}
\|Jf\|_{L^1(\nu)}=T_H(\mu_{2\e},\mu_{2\e},\mu_{2\e},f,\Delta_\e).
\end{equation}
By (\ref{L1TH}) and Theorems \ref{deforestation} and \ref{LinfL1}, we have 
\[
\|J\|_{L^{\infty}\to L^1(\nu)}\lesssim \e^{-2(d-s)+\frac{(d+2-s)(d-2)}{2(2d-s)}}.
\]
From the bound $\|\mu_\e\|_{L^\infty}\lesssim \e^{-(d-s)}$, we have 
\[
\|J\|_{L^\infty\to L^\infty}\lesssim \e^{-(d-s)}.
\]
By Riesz-Thorin, it follows that
\[
\|J\|_{L^{\infty}\to L^2(\nu)}\lesssim \left(\e^{-2(d-s)+\frac{(d+2-s)(d-2)}{2(2d-s)}}\right)^{1/2}\left(\e^{-(d-s)}\right)^{1/2}=\e^{-\frac{3}{2}(d-s)+\frac{(d+2-s)(d-2)}{4(2d-s)}}.
\]
This, together with (\ref{L2TG}) and the bound $\|\mu_\e\|_{L^\infty}\lesssim \e^{-(d-s)}$, gives
\begin{align*}
T_G\mu_{2\e}-T_G\mu_\e &\lesssim \|J\mu_{2\e}\|_{L^2(\nu)}^2 \\
&\lesssim \left(\e^{-\frac{3}{2}(d-s)+\frac{(d+2-s)(d-2)}{4(2d-s)}}\right)^2\|\mu_{2\e}\|_{L^\infty}^2 \\
&\lesssim \e^{-5(d-s)+\frac{(d+2-s)(d-2)}{2(2d-s)}}.
\end{align*}
If $s_d<s\leq d$, then
\[
-5(d-s)+\frac{(d+2-s)(d-2)}{2(2d-s)}>0.
\]
Therefore, we have an upper bound on $T_G\mu_\e$ by the telescoping identity $T_G\mu_\e=\sum_j T_G\mu_{2^{-j}\e}-T_G\mu_{2^{-j-1}\e}$.
\end{proof}
For technical reasons which will arise later, we will also need the following related bound.
\begin{lemma}
\label{Bupperbound}
Let $d\geq 3$, and let $B$ be the graph shown in Figure \ref{book}.  If $s>s_d$, then $|T_B\mu_\e|\lesssim 1$.
\end{lemma}

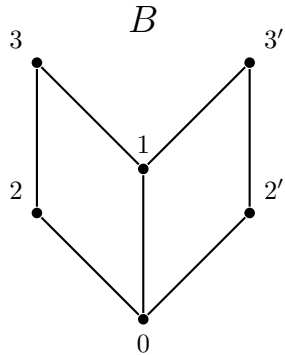
\begin{figure}[h!]

  \centering
\begin{tikzpicture}[scale=2, thick]
  \node[fill=black, circle, minimum size=4pt, inner sep=0pt, label=below:$0$] (y123) at (0,0) {};
  \node[fill=black, circle, minimum size=4pt, inner sep=0pt, label=above:$1$] (x2) at (0,1) {};
  \node[fill=black, circle, minimum size=4pt, inner sep=0pt, label=above left:$2$] (x1) at (-.707,.707) {};
  \node[fill=black, circle, minimum size=4pt, inner sep=0pt, label=above right:$2'$] (x3) at (.707,.707) {};
  \node[fill=black, circle, minimum size=4pt, inner sep=0pt, label=above left:$3$] (y12) at (-.707,1.707) {};
  \node[fill=black, circle, minimum size=4pt, inner sep=0pt, label=above right:$3'$] (y23) at (.707,1.707) {};
 \node at (0,2) {\LARGE $B$};

  \draw (y123) -- (x1);
  \draw (y123) -- (x2);
  \draw (y123) -- (x3);
  \draw (y12) -- (x1);
  \draw (y12) -- (x2);

    \draw (y23) -- (x2);
    \draw (y23) -- (x3);

\end{tikzpicture}

\caption{The "book" graph $B$.}\label{book}

\end{figure}

\begin{proof}
The proof is essentially the same as the proof of Lemma \ref{mainupperbound}.  This time, we consider the measure
\[
d\nu(x,u_1)=\mu_{2\e}(x)\mu_{2\e}(x-u_1)d\sigma(u_1)dx
\]
and the operator
\[
Jf(x,u_1)=\int\int f(x-u_2)\Delta_\e(x-u_3).
\]
With this choice, we have
\[
\|Jf\|_{L^2(\nu)}^2=T_B(\mu_{2\e},\dots,\mu_{2\e},f,\Delta_\e)
\]
and
\[
\|Jf\|_{L^1(\nu)}=T_\Gamma(\mu_{2\e},\mu_{2\e},\mu_{2\e},f,\Delta_\e),
\]
where $\Gamma$ denotes the $4$-cycle graph (i.e., the graph $H$ with the degree $1$ vertex removed).  The same interpolation argument as in the proof of Lemma \ref{mainupperbound}, with the same bounds, yields the result.
\end{proof}
\subsection{The lower bound}
\label{lowerbound}
The lower bound will be proved by making a series of reductions.  We have two basic tools at our disposal.  First, applying Lemma \ref{deforestation}, we may remove any vertices of degree 1 (or 0), and it is equivalent to bound the integral corresponding to the remaining graph.  We combine this with a Cauchy-Schwarz argument that allows us to exploit the symmetry of the graphs which arise when considering VC-dimension.  Before getting to the estimates, we first describe the symmetry of $G$ which we will exploit repeatedly.  The vertices of $G$ can be partitioned into parts $\{0,1,2\}\cup\{3,4\}\cup\{5,6\}$, and the adjacencies of $\{3,4\}$ with $\{1,2,3\}$ and each other are mirrored by $\{5,6\}$.  Therefore, the graph $G$ can be viewed as two copies of the graph $H$, ``glued together'' on the common vertices $\{1,2,3\}$.  The graph $H$ has a similar symmetry, once we have removed the vertex of degree 1.  It can be identified with two copies of the $2$-chain, glued together at the endpoints.
\begin{lemma}[Reduction from $G$ to $H$]
\label{lowerboundGtoH}
Let $G$ and $H$ be the graphs defined in the beginning of this section (and shown in Figure \ref{shatter3}), and let $\mu$ be a compactly supported probability measure.  For any $c>0$ and any $t,\e>0$ such that $\sigma_t^\e*\mu\approx 1$ on the support of $\mu$, we have
\[
\Lambda_{G,t,c}^\e\mu\gtrsim \left(\Lambda_{H,t,c}^\e\mu\right)^2.
\]
Here, $\Lambda$ is the operator defined in (\ref{Lambda}).
\end{lemma}
\begin{proof}
Note that the definition of $\Lambda_{G,t,c}^\e$ depends on the ``non-degenerate set'' $N_c$ defined in (\ref{Nc}), and this in turn depends on the underlying graph.  We will denote by $N_{G,c}$ and $N_{H,c}$ the sets corresponding to the graphs $G$ and $H$, respectively.  The symmetry described above yields the following relation between the two sets:
\[
N_{G,c}=\{(x_0,\dots,x_6)\in E^7: (x_0,x_1,x_2,x_3,x_4)\in N_{H,c} \text{ and } (x_0,x_1,x_2,x_5,x_6)\in N_{H,c}\}.
\]
For fixed $x_0,x_1,x_2$, let 
\[
N_c(x_0,x_1,x_2)=\{(x_3,x_4)\in E^2: (x_0,\dots,x_4)\in N_{H,c}\}.
\]
The relation can then be written
\[
N_{G,c}=\{(x_0,\dots,x_6)\in E^7: (x_3,x_4)\in N_c(x_0,x_1,x_2) \text{ and } (x_5,x_6)\in N_c(x_0,x_1,x_2)\}.
\]
Finally, define a measure $\nu_t^\e$ by
\[
d\nu_t^\e(x_0,x_1,x_2)=\sigma_t^\e(x_1-x_0)d\mu(x_0)d\mu(x_1)d\mu(x_2),
\]
and define a function $f_t^\e$ by
\[
f_t^\e(x_0,\dots,x_4)=\sigma_t^\e(x_3-x_0)\sigma_t^\e(x_1-x_0)\sigma_t^\e(x_4-x_3)\sigma_t^\e(x_4-x_1).
\]
We have
\[
\Lambda_{G,t,c}^\e\mu=\int\left(\int\int_{N_c(x_0,x_1,x_2)} f_t^\e(x_0,\dots,x_4)d\mu(x_3)d\mu(x_4)\right)^2d\nu_t^\e(x_0,x_1,x_2).
\]
and
\[
\Lambda_{H,t,c}^\e\mu=\int\int\int_{(x_3,x_4)\in N_c(x_0,x_1,x_2)} f_t^\e(x_0,\dots,x_4)d\mu(x_3)d\mu(x_4)d\nu_t^\e(x_0,x_1,x_2).
\]
Since $\sigma_t^\e*\mu\approx 1$ on the support of $\mu$, and since $\mu$ is a probability measure, we conclude $\lambda(\R^3)\approx 1$.  Therefore, by Cauchy-Schwarz, we have
\begin{align*}
\left(\Lambda_{H,t,c}^\e\mu\right)^2&=\left(\int\int\int_{(x_3,x_4)\in N_c(x_0,x_1,x_2)} f_t^\e(x_0,\dots,x_4)d\mu(x_3)d\mu(x_4)d\nu_t^\e(x_0,x_1,x_2)\right)^2 \\
&\leq \left(\int d\nu_t^\e(x_0,x_1,x_2)\right)\left(\int\left(\int\int_{N_c(x_0,x_1,x_2)} f_t^\e(x_0,\dots,x_4)d\mu(x_3)d\mu(x_4)\right)^2d\nu_t^\e(x_0,x_1,x_2)\right) \\
&\approx \Lambda_{G,t,c}^\e\mu
\end{align*}
\end{proof}
\begin{lemma}[Lower bound for $4$-cycle]
\label{lowerbound4cycle}
Let $\Gamma$ denote the $4$-cycle graph; that is, the graph with vertices $\{0,1,2,3\}$ and adjacency relation $0\sim 1\sim 3\sim 2\sim 0$.  Let $\mu$ be a compactly supported Frostman probability measure on $\R^d$ with exponent $s>1$, and let $t,\e>0$ 
.  There exists $c>0$ (depending on the support of $\mu$) such that, if $\e$ is sufficiently small (depending on $c$), then
\[
\Lambda_{\Gamma,t,c}^\e\mu\gtrsim_c 1.
\]
\end{lemma}
\begin{proof}
Let $E\subset \R^d$ be the support of $\mu$.  Consider a covering of $E$ by a finite collection of tubes parallel to the $x_1$-axis of diameter $c$ and finite length.  There must be at least one tube $T$ such that $\mu(E\cap T)>0$.  Let $2T$ denote the tube with the same axis and diameter $2c$.  Since $2T$ can be covered by $O(c^{-1})$ balls of radius $c$, we also have $\mu(2T)\lesssim c^{s-1}$.  Since $s>1$, if $c$ is sufficiently small, we have $\mu(2T)<1$.  Let $A=E\cap T$, let $B=E\setminus 2T$, and let $\mu_A,\mu_B$ the the restrictions of $\mu$ to $A$ and $B$, respectively, and normalized to be probability measures.  If $N_c$ is the non-degeneracy set corresponding to the graph $\Gamma$, then we have
\[
N_c\supset \{(x_0,x_1,x_2,x_3):x_1,x_2\in A,x_0,x_3\in B\},
\]
hence
\begin{align*}
\Lambda_{\Gamma,t,c}^\e\mu&\gtrsim_c \int\int\int\int \sigma_t^\e(x_1-x_0)\sigma_t^\e(x_2-x_0)\sigma_t^\e(x_3-x_1)\sigma_t^\e(x_3-x_2)d\mu_B(x_0)d\mu_A(x_1)d\mu_A(x_2)d\mu_B(x_3) \\
&=\int\int \left(\int \sigma_t^\e(x_1-x_0)\sigma_t^\e(x_2-x_0)d\mu_B(x_0)\right)^2d\mu_A(x_1)d\mu_A(x_2).
\end{align*}
Next, observe that for any $x\in A$, the support of the function $y\mapsto \sigma_t^\e(x-y)$ can be covered by $c^{d-1}$-many balls of radius $\e$, and hence has measure $\lesssim c^{-(d-1)}\e^s$.  Therefore, for each $x\in A$, we have
\[
\sigma_t^\e*\mu|_{2T}(x)=\int_{2T}\sigma_t^\e(x-y)d\mu(y)\lesssim c^{-(d-1)}\e^{s-1},
\]
which can be made arbitrarily small by taking $\e$ small (depending on $c$).  Since
\[
\sigma_t^\e*\mu_B>\sigma_t^\e*\mu|_B=\sigma_t^\e*\mu-\sigma_t^\e*\mu|_{2T},
\]
we conclude that $\sigma_t^\e*\mu_B(x)\gtrsim 1$ for all $x\in A$, provided $\e$ is sufficiently small.  This, together with Cauch-Schwarz, gives 
\begin{align*}
1&\lesssim \left(\int \sigma_t^\e*\mu_B(x_0)^2d\mu_A(x_0)\right)^2 \\
&=\left(\int\int\int \sigma_t^\e(x_1-x_0)\sigma_t^\e(x_2-x_0)d\mu_B(x_0)d\mu_A(x_1)d\mu_A(x_2)\right)^2 \\
&\leq \left(\int\int d\mu_A(x_1)d\mu_A(x_2)\right)\left(\int\int \left(\int \sigma_t^\e(x_1-x_0)\sigma_t^\e(x_2-x_0)d\mu_B(x_0)\right)^2d\mu_A(x_1)d\mu_A(x_2)\right) \\
&\lesssim_c \Lambda_{\Gamma,t,c}^\e\mu.
\end{align*}
\end{proof}
\begin{lemma}
\label{mainlowerbound}
Let $d\geq 2$, and let $E\subset \R^d$ be a compact set satisfying $\dim E>\frac{d+1}{2}$.  For any $\frac{d+1}{2}<s<\dim E$ there exists a Frostman probability measure $\mu$ of exponent $s$ supported on a subset of $E$, a constant $c>0$, and an interval $I\subset (0,\infty)$ such that for all $t\in I$ and $\e>0$, we have
\[
\Lambda_{\mathcal{G}_3,t,c}^\e\mu \gtrsim_c 1.
\]
\begin{proof}
Let $\mu_0$ be any Frostman measure of exponent $s$.  By Lemma \ref{IT19}, there is an interval $I$ and a subset of $E$ on which $\sigma_t^\e*\mu\approx 1$ for all $t\in I$.  Let $\mu$ be the restriction of $\mu_0$ to this set, normalized to be a probability measure.  By Lemmas \ref{deforestation}, \ref{lowerboundGtoH}, and \ref{lowerbound4cycle} we have
\[
\Lambda_{\mathcal{G}_3,t,c}^\e \mu\approx \Lambda_{G,t,c}^\e\gtrsim \Lambda_{H,t,c}^\e\approx \Lambda_{\Gamma,t,c}^\e\gtrsim_c 1,
\]
where the graphs $G,H,\Gamma$ are as in the statement of those lemmas.
\end{proof}
\end{lemma}
\subsection{Proofs of Theorems}
\label{proofmt}
The main work in proving our theorems is contained in the upper and lower bounds for the various operators studied in this section.  What remains is to show that the resulting configurations are not ``degenerate'', in the sense that we do not have points which are supposed to be distinct but are not, or that we do not have distance $t$ between pairs of points which do not correspond to an edge of our graph.  The following calculation will be needed to handle the case $d=3$.
\begin{lemma}
\label{smallballconvolution}
Let $t,\e>0$, let $d\geq 3$, and let $B\subset \R^d$ be a ball of radius $\e$.  Let $\mu$ be a compactly supported Frostman probability measure such that the $s$-energy (\ref{energy}) is finite.  If $s>\frac{d+5}{3}$, then there exists $\delta>0$ such that for any $x\in \R^d$, we have
\[
\int_B \sigma_t^\e(x-y)d\mu(y)\lesssim \e^{2+\delta}.
\]
\end{lemma}
\begin{remark}
The integral could be bounded trivially by $\e^{-1+s}$.  If $d\geq 4$, then in application we will always have $s>3$, and the result follows immediately.  The lemma is needed to handle the $d=3$ case.
\end{remark}
\begin{proof}
If $\mu(B)=0$, there is nothing to prove.  Otherwise, define
\[
\nu=\frac{\mu|_B}{\mu(B)}.
\]
Then, $\nu$ is a Frostman probability measure of exponent $s$ supported on $B$, and
\[
\int_B \sigma_t^\e(x-y)d\mu(y)=\int \sigma_t^\e(x-y)d\mu|_B(y)\lesssim \e^s\int \sigma_t^\e(x-y)d\nu(y).
\] 
Recall $\sigma_t^\e=\f_\e*\sigma_t$.  By Plancherel, we get
\begin{equation}
\label{plancherelbound}
\int_B \sigma_t^\e(x-y)d\mu(y)\lesssim \e^s\int \widehat{\f}(\e\xi)(1+|\xi|)^{-\frac{d-1}{2}}\widehat{\nu}(\xi)d\xi.
\end{equation}
We estimate the integral on the right in the ranges $0<|\xi|<1$, $1<|\xi|<\e^{-\alpha}$, and $|\xi|>\e^{-\alpha}$ separately, where $\alpha>1$ is a parameter to be determined later.  In the first range, the integral is clearly bounded, and the factor of $\e^s$ on the right hand side of (\ref{plancherelbound}) is better than the bound claimed in the theorem.  For the second range, we have
\begin{align*}
\int_{1<|\xi|<\e^{-\alpha}} \widehat{\f}(\e\xi)(1+|\xi|)^{-\frac{d-1}{2}}\widehat{\nu}(\xi)d\xi&\lesssim \int_{1<|\xi|<\e^{-\alpha}} |\xi|^{-\frac{d-1}{2}}|\widehat{\nu}(\xi)|d\xi \\
&=\int_{1<|\xi|<\e^{-\alpha}} |\widehat{\nu}(\xi)||\xi|^{-\frac{(d-s)}{2}}|\xi|^{\frac{d-s}{2}-\frac{d-1}{2}}d\xi \\
&\leq\left(\int_{1<|\xi|<\e^{-\alpha}} |\widehat{\nu}(\xi)|^2|\xi|^{-(d-s)}d\xi\right)^{1/2}\left(\int_{1<|\xi|<\e^{-\alpha}} |\xi|^{-(s-1)}\widehat{\nu}(\xi)d\xi\right)^{1/2} \\
&\lesssim \e^{-\frac{\alpha(d-s+1)}{2}}.
\end{align*}
Plugging this into (\ref{plancherelbound}) gives the claimed bound.  For the third range, we use the fact that $\widehat{\f}$ is Schwarz, and therefore satisfies $|\widehat{\f}(\xi)\lesssim_N |\xi|^{-N}$ for any $N$.  We have
\begin{align*}
\int_{|\xi|>\e^{-\alpha}} \widehat{\f}(\e\xi)(1+|\xi|)^{-\frac{d-1}{2}}\widehat{\nu}(\xi)d\xi&\lesssim_N \e^{-N}\int_{|\xi|>\e^{-\alpha}} |\xi|^{-N-\frac{d-1}{2}}d\xi \\
&=\e^{-N}(\e^{-\alpha})^{-N+\frac{d+1}{2}+1}.
\end{align*}
Combining each estimate with (\ref{plancherelbound}), we get
\[
\int_B \sigma_t^\e(x-y)d\mu(y)\lesssim_\alpha \e^s+\e^{s-\frac{\alpha(d-s+1)}{2}}+\e^{s+(\alpha-1)N-\alpha(\frac{d+1}{2}+1)}.
\]
Since $s>2$, the first bound is sufficient.  We will have $s-\frac{\alpha(d-s+1)}{2}>2$ provided $s>\frac{d+5}{3}$ and $\alpha$ sufficiently close to $1$.  Finally, the exponent in the third term can be made as large as desired by choosing $N$ sufficiently large (depending on $\alpha$ and $d$).
\end{proof}
To complete the proof of our main theorem (Theorem \ref{main}), it will be convenient to return to our original notation for $\mathcal{G}_3$: the vertex set is $\{1,2,3\}\cup \mathcal{P}(\{1,2,3\})$, and there is an edge between $i\in \{1,2,3\}$ and $I\subset \{1,2,3\}$ if and only if $i\in I$.  When embedding the graph in $\R^d$, we will always denote the points corresponding to $i\in\{1,2,3\}$ by $x_i$ and those corresponding to $I\subset \{1,2,3\}$ by $y_I$.  To simplify notation, we will write things like $y_{12}$ instead of $y_{\{1,2\}}$.
\begin{proof}[Proof of Theorem \ref{main}]
Let $\mu$ be a Frostman measure supported on $E$ with exponent $s>s_d$.  Passing to a subset if necessary (by Lemma \ref{IT19}), we may assume without loss of generality that there exists an interval $I\subset (0,\infty)$ such that for all $t\in I$ and all $\e>0$, we have $\sigma_t^\e*\mu\approx 1$ on $E$.  For any such $t$ we have, by Lemmas \ref{deforestation} and \ref{mainupperbound}, the bound $\Lambda_{\mathcal{G}_3,t,c}^\e\lesssim 1$ for all $c,\e>0$.  By Lemma \ref{mainlowerbound}, we have $\Lambda_{\mathcal{G}_3,t,c}^\e \gtrsim_c 1$ for some fixed $c>0$.  By Lemma \ref{Borelmeasurelemma}, 
the measures $\Lambda_{\mathcal{G}_3,t,c}^\e\mu$ converge to a weak* limit $\Lambda$.  In order to show that the VC-dimension of $\mathcal{H}_t(E)$ is at least $3$, we must show that there exists a set of points
\[
\{x_i:i\in\{1,2,3\}\}\cup\{y_I:I\subset \{1,2,3\}\}\subset E
\]
such that:
\begin{enumerate}[(i)]
\item $|x_i-y_I|=t$ if $i\in I$,
\item $|x_i-y_I|\neq t$ if $i\notin I$,
\item The points $x_i,y_I$ are all distinct.
\end{enumerate}
Since $\Lambda$ is a probability measure supported on
\[
\{(x_i,y_I)_{i,I}:|x_i-y_I|=t \text{ if } i\in I\},
\]
it is immediate that there exist points $\{x_i,y_I\}_{i,I}$ satisfying (i).  To complete the proof, we show that the set of points failing to satisfy (ii) or (iii) has measure zero.  Recall that by definition, for any set of configurations $A$, we have
\[
\Lambda\mu(A)\approx \int\cdots\int_{N_c\cap A_\e} \left(\prod_{i\in I} \sigma_t^\e(x_i-y_I)\right)d\mu(x_1)\cdots d\mu(y_{123}).
\]
Since $\sigma_t^\e*\mu\approx 1$ on $E$, we can immediately integrate out all degree $0$ and $1$ vertices.  Henceforth, we will assume that the integral above is over only the points $x_1,x_2,x_3,y_{12},y_{23},y_{13},y_{123}$.  

First, let $A$ be the set of configurations for which $x_1=x_2$.  We can use $L^\infty$ bounds on the terms $\sigma_t^\e(x_1-y_{12})$ and $\sigma_t^\e(x_1-y_{13})$, giving
\[
\Lambda\mu(A)\lesssim \e^{-2}\int\cdots\int \left(\prod_{i\in I, i\neq 1} \sigma_t^\e(x_i-y_I)\right)\left(\int_{|x_1-x_2|<\e}\sigma_t^\e(x_1-y_{123})d\mu(x_1)\right)d\mu(x_2)\cdots d\mu(y_{123}).
\]
Letting $\delta$ be as in Lemma \ref{smallballconvolution}, this is
\[
\Lambda\mu(A)\lesssim \e^\delta \int\cdots\int \left(\prod_{i\in I, i\neq 1} \sigma_t^\e(x_i-y_I)\right)d\mu(x_2)\cdots d\mu(y_{123}).
\]
After applying Lemma \ref{deforestation}, the integral on the right reduces to the configuration integral for the $4$-cycle $\Gamma$, which is bounded by Lemma \ref{4cycleupperbound}.  Any of the degeneracies $x_i=x_j$ or $x_i=y_I$ can be handled the same way.  Degeneracies of the form $y_I=y_J$ can also be handled similarly; the only change is that after the vertex is eliminated, the resulting integral is the configuration integral for the graph $B$ in Lemma \ref{Bupperbound} (assuming the point deleted is not $y_{123}$), which is bounded by that lemma.  The case of $y_{123}$ can be ignored, as one may always choose to delete whatever vertex it coincides with, reducing to one of the cases above.  This shows that the set of configurations failing condition (iii) has measure zero; it remains to handle those violating (ii).  This case is similar to the above, except after a vertex is deleted we are integrating on an annulus rather than a ball.  On the other hand, we do not need to consider any degree $3$ vertices.  This is because $y_{123}$ is already distance $t$ from all of $x_1,x_2,x_3$, so there are no potential ``extra'' edges to worry about.  Let $A$ be the set of configurations where $|x_1-y_{23}|=t$.  By applying the $L^\infty$ bound to the terms $\sigma_t(x_2-y_{23})$ and $\sigma_t^\e(x_3-y_{23})$, we get
\[
\Lambda\mu(A)\lesssim\e^{-2}\int\cdots\int \left(\prod_{i\in I, I\neq \{2,3\}} \sigma_t^\e(x_i-y_I)\right)\left(\int_{\substack{||x_1-y_{23}|-t|<\e \\  ||x_2-y_{23}|-t|<\e \\ ||x_3-y_{23}|-t|<\e}}d\mu(y_{23})\right)d\mu(x_1)\cdots d\mu(y_{123}).
\]
The inner integral is over an intersection of three spheres, which we may assume are distinct since we have already shown the sets where $x_i=x_j$ are measure zero.  This surface can be covered by $\e^{-(d-3)}$-many balls of radius $\e$, and thus has measure $\lesssim \e^{-(d-3)+s}$.  When the point $y_{23}$ is deleted, the resulting integral again reduces to the one corresponding to the graph $B$, which is bounded by Lemma \ref{Bupperbound}.  The result is the bound
\[
\Lambda\mu(A)\lesssim\e^{-2}\e^{-(d-3)+s}=\e^{s-(d-1)}.
\]
Since $s>d-1$, we conclude $\Lambda\mu(A)=0$.  From this, we conclude that the set $E$ contains the necessary point configuration.
\end{proof}
\begin{proof}[Proof of Theorem \ref{4cycle}]
Let $\mu$ be a Frostman probability measure supported on $E$ of exponent $s$.  If $s$ satisfies the threshold assumed in the statement of the theorem, then it satisfies (\ref{4cycleexponent}) with $d=3$.  By Lemmas \ref{4cycleupperbound} and \ref{lowerbound4cycle}, we have $\Lambda_{\Gamma,t,c}\approx 1$ for some fixed $c>0$ and all $t$ in some interval.  Let $\Lambda\mu$ be the measure given by Lemma \ref{Borelmeasurelemma}.  The support of this measure is
\[
\{(x,y,z,w)\in E^4\cap N_c:|x-y|=|y-z|=|z-w|=|w-x|=t\}.
\]
For any configuration $(x,y,z,w)$ in the above set, the distance equations prevent the relations $x=y$ and $x=w$.  If $A$ is the set of configurations such that $x=z$, then
\begin{align*}
\Lambda\mu(A)&\approx\int\int\int\int_{|x-z|<\e}\sigma_t^\e(x-y)\sigma_t^\e(y-z)\sigma_t^\e(z-w)\sigma_t^\e(w-x)d\mu(x)d\mu(y)d\mu(z)d\mu(w) \\
&\lesssim \e^{-2}\int\int\int\sigma_t^\e(y-z)\sigma_t^\e(z-w)\left(\int_{|x-z|<\e}d\mu(x)\right)d\mu(y)d\mu(z)d\mu(w) \\
&\lesssim \e^{s-2}.
\end{align*}
Since $s>2$, we conclude $\Lambda\mu(A)=0$.
\end{proof}
\begin{proof}[Proof of Theorem \ref{mainchain}]
The graph $\mathcal{G}_2$ consists of a $4$-chain and a single isolated vertex, as in Figure \ref{shatter2}

\begin{figure}[h!]
\begin{tikzpicture}[scale=2, thick]

  \node[fill=black, circle, minimum size=4pt, inner sep=0pt, label=above :$x_2$] (x2) at (.707,.707) {};
  \node[fill=black, circle, minimum size=4pt, inner sep=0pt, label=above:$x_1$] (x1) at (-.707,.707) {};

  \node[fill=black, circle, minimum size=4pt, inner sep=0pt, label=below:$y_{12}$] (y12) at (0,0) {};

  \node[fill=black, circle, minimum size=4pt, inner sep=0pt, label=above left:$y_1$] (y1) at (-1.41,0) {};

  \node[fill=black, circle, minimum size=4pt, inner sep=0pt, label=above right:$y_2$] (y2) at (1.41,0) {};
  \node[fill=black, circle, minimum size=4pt, inner sep=0pt, label= left:$y_\varnothing$] (y0) at (-1.41,.8) {};
  \node at (0,1) {\LARGE $\mathcal{G}_2$};


  \draw (y12) -- (x1);
  \draw (y12) -- (x2);
 
    \draw (y1) -- (x1);
    \draw (y2) -- (x2);

\end{tikzpicture}
\caption{The 2-shattering graph}\label{shatter2}
\end{figure}
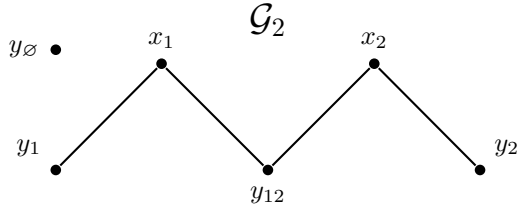

We will continue our notational conventions from the proof of Theorem \ref{main}.  Thus, we will call our points $x_1,x_2$ and $y_1,y_2,y_{12},y_\varnothing$.  The boundedness of the configuration integral of this graph follows immediately from Lemma \ref{IT19}; let $\Lambda\mu$ be the limit measure given by Lemma \ref{Borelmeasurelemma}.  To complete the proof, we need to exclude the case where some vertices coincide, or where we have extra edges.  To start, let $A$ be the set of tuples where $x_1=x_2$.  Then,
\begin{align*}
\Lambda\mu(A)&\approx \int\int\int\int\int_{|x_1-x_2|<\e}\sigma_t^\e(y_1-x_1)\sigma_t^\e(x_1-y_{12})\sigma_t^\e(y_{12}-x_2)\sigma_t^\e(x_2-y_2)d\mu(x_1)d\mu(x_2)d\mu(y_1)d\mu(y_2)d\mu(y_{12}) \\
&\approx \int\int\int_{|x_1-x_2|<\e}\sigma_t^\e(x_1-y_{12})\sigma_t^\e(y_{12}-x_2) d\mu(x_1)d\mu(x_2)d\mu(y_{12}) \\
&\e^{-1}\int\int\int_{|x_1-x_2|<\e}\sigma_t^\e(x_1-y_{12})d\mu(x_1)d\mu(x_2)d\mu(y_{12}) \\
&\approx \e^{-1}\int\int_{|x_1-x_2|<\e}d\mu(x_1)d\mu(x_2)  \\
&\lesssim \e^{s-1}.
\end{align*}
Since $s>1$, we conclude $\Lambda\mu(A)=0$.  All other cases of points coinciding are similar.  To handle possible extra edges, now consider the case where $A'$ consists of configurations with $|x_1-y_2|=t$.  We would then have
\begin{align*}
\Lambda\mu(A')&\approx \int\int\int\int\int_{||x_1-y_2|-t|<\e}\sigma_t^\e(y_1-x_1)\sigma_t^\e(x_1-y_{12})\sigma_t^\e(y_{12}-x_2)\sigma_t^\e(x_2-y_2)d\mu(x_1)d\mu(x_2)d\mu(y_1)d\mu(y_2)d\mu(y_{12}) \\
&\approx \int\int\int\sigma_t^\e(y_{12}-x_2)\sigma_t^\e(x_2-y_2)\left(\int_{||x_1-y_2|-t|<\e}\sigma_t^\e(x_1-y_{12})d\mu(x_1)\right)d\mu(x_2)d\mu(y_2)d\mu(y_{12}) \\
&\lesssim \e^{-1}\int\int\int\sigma_t^\e(y_{12}-x_2)\sigma_t^\e(x_2-y_2)\left(\int_{\substack{||x_1-y_2|-t|<\e \\ ||x_1-y_{12}|-t|<\e}}d\mu(x_1)\right)d\mu(x_2)d\mu(y_2)d\mu(y_{12}) \\
&\lesssim \e^{s-(d-1)}.
\end{align*}
If $s>d-1$, then $\Lambda\mu(A')=0$.  We conclude that $E$ contains the required configuration.
\end{proof}

\bibliographystyle{plain}
\bibliography{EuclideanVCrefs}

\end{document}